\documentclass[twoside,english]{amsart}
\usepackage[T1]{fontenc}
\usepackage[latin9]{inputenc}
\usepackage[a4paper]{geometry}
\usepackage{babel}
\usepackage{amstext}
\usepackage{amsthm}
\usepackage{amssymb}
\usepackage{setspace}
\usepackage[colorlinks=true,linkcolor=black,anchorcolor=black,citecolor=black,filecolor=black,menucolor=black,runcolor=black,urlcolor=black]{hyperref}

\makeatletter
\numberwithin{equation}{section}
\numberwithin{figure}{section}
\theoremstyle{plain}
\newtheorem{thm}{\protect\theoremname}
\theoremstyle{plain}
\newtheorem{lem}[thm]{\protect\lemmaname}
\theoremstyle{plain}
\newtheorem{cor}[thm]{\protect\corollaryname}

\makeatother

\providecommand{\corollaryname}{Corollary}
\providecommand{\lemmaname}{Lemma}
\providecommand{\theoremname}{Theorem}

\begin{document}
\title{Coercive Inequalities and U-Bounds}
\author{E. Bou Dagher}
\address{Esther Bou Dagher: \newline Department of Mathematics \newline Imperial College London \newline 180 Queen's Gate, London SW7 2AZ \newline United Kingdom}
\email{esther.bou-dagher17@imperial.ac.uk}
\author{B. Zegarli\'{n}ski}
\address{Bogus\l{}aw Zegarli\'{n}ski: \newline Department of Mathematics \newline Imperial College London \newline 180 Queen's Gate, London SW7 2AZ \newline United Kingdom}
\email{b.zegarlinski@imperial.ac.uk}

\providecommand{\keywords}[1]{\textbf{\textit{Keywords---}} #1}

\begin{abstract}
We prove Poincaré and Log$^{\beta}$-Sobolev inequalities for probability
measures on step-two Carnot groups.

\tableofcontents{}

\end{abstract}

\maketitle

\section{Introduction}

\begin{onehalfspace}
Although the question of obtaining coercive inequalities such as the
Poincaré or the Logarithmic Sobolev inequalities for a probability
measure on a metric measure space has been a subject of numerous works,
the literature on this topic in the setup Carnot groups is scarce.

In \cite{key-5}, L. Gross obtained the following Logarithmic Sobolev
inequality:

\begin{equation}
\int_{\mathbb{R}^{n}}f^{2}log\left(\frac{f^{2}}{\int_{\mathbb{R}^{n}}f^{2}d\mu}\right)d\mu\leq2\int_{\mathbb{R}^{n}}\vert\triangledown f\vert^{2}d\mu,\label{eq:1}
\end{equation}
where ${\displaystyle \triangledown}$ is the standard gradient on
$\mathbb{R}^{n}$ and ${\displaystyle d\mu=\frac{e^{-\frac{\vert x\vert^{2}}{2}}}{Z}d\lambda}$
is the Gaussian measure.
\end{onehalfspace}

In a setup of a more general metric space, a natural question would
be to try to find similar inequalities with different measures of
the form ${\displaystyle d\mu=\frac{e^{-U\left(d\right)}}{Z}d\lambda},$
where $U$ is a function of a metric $d,$ and where the Euclidean
gradient is replaced by a more general sub-gradient in $\mathbb{R}^{n}.$

\begin{onehalfspace}
Aside from their theoretical importance, such inequalities are needed
because of their applications, some of which will be discussed briefly.
L.Gross also pointed out (\cite{key-5}) the importance of the inequality
(\ref{eq:1}) in the sense that it can be extended to infinite dimensions
with additional useful results. (See also works: \cite{key-13,key-3,key-16,key-14,key-12,key-11,key-15}.)
He proved that if $\mathcal{L}$ is the non-positive self-adjoint
operator on $L^{2}\left(\mu\right)$ such that
\[
\left(-\mathcal{L}f,f\right)_{L^{2}\left(\mu\right)}=\int_{\mathbb{R}^{n}}\vert\triangledown f\vert^{2}d\mu,
\]
then (\ref{eq:1}) is equivalent to the fact that the semigroup ${\displaystyle P_{t}=e^{t\mathcal{L}}}$
generated by $\mathcal{L}$ is hypercontractive: i.e. for $q\left(t\right)\leq1+\left(q-1\right)e^{2t}$
with $q>1$, we have ${\displaystyle \parallel P_{t}f\parallel_{q\left(t\right)}\leq\parallel f\parallel_{q}}$
for all $f\in L^{q}\left(\mu\right).$ (\cite{key-5})

In \cite{key-1}, D. Bakry and M. Emery extended the Logarithmic Sobolev
inequality for a larger class of probability measures defined on Riemaniann
manifolds under an important Curvature-Dimension condition. More generally,
if $\left(\Omega,F,\mu\right)$ a probability space, and $\mathcal{L}$
is a non-positive self-adjoint operator acting on $L^{2}\left(\mu\right),$
we say that the measure $\mu$ satisfies a Logarithmic Sobolev inequality
if there is a constant $c$ such that, for $f\in D\left(\mathcal{L}\right),$ 

\[
\int f^{2}log\frac{f^{2}}{\int f^{2}d\mu}d\mu\leq c\int f\left(-\mathcal{L}f\right)d\mu,
\]

In this general setting, the connection between this inequality and
the property of hypercontractivity was shown in \cite{key-5}. 

Another generalisation, the so-called q-Logarithmic Sobolev inequality,
in the setting of a metric measure space, was obtained by S. Bobkov
and M. Ledoux in \cite{key-2}, in the form: 

\[
\int f^{q}log\frac{f^{q}}{\int f^{q}d\mu}d\mu\leq c\int\vert\triangledown f\vert^{q}d\mu,
\]
where $q\in\left(1,2\right].$ Here, on a metric space, the magnitude
of the gradient is defined by

\[
\vert\triangledown f\vert\left(x\right)=\underset{d\left(x,y\right)\rightarrow0}{limsup}\frac{\vert f\left(x\right)-f\left(y\right)\vert}{d\left(x,y\right)}.
\]

In \cite{key-3}, S. Bobkov and B. Zegarli\'{n}ski showed that the
q-Logarithmic Sobolev inequality is better than the classical $q=2$
inequality in the sense that one gets a stronger decay of tail estimates.
In addition, when the space is finite, and under weak conditions,
they proved that the corresponding semigroup $P_{t}$ is ultracontractive
i.e. 

\[
\parallel P_{t}f\parallel_{\infty}\leq\parallel f\parallel_{p}
\]
for all $t\geq0$ and $p\in\left[1,\infty\right).$

We point out, that in \cite{key-9}, M. Ledoux made a connection between
the Logarithmic Sobolev inequality and the isoperimetric problem.
(See also: \cite{key-29,key-8,key-1})

The important q-Poincaré inequality

\[
\int\left|f-\int fd\mu\right|^{q}d\mu\leq c\int\vert\triangledown f\vert^{q}d\mu
\]

can be obtained from the q-Logarithmic Sobolev inequality by simply
replacing $f$ by $1+\varepsilon f$ in that inequality, and letting
$\varepsilon\rightarrow0.$

In this paper, our primary interest is to prove the existence of coercive
inequalities for different measures in the setting of step-two nilpotent
Lie groups, whose tangent space at every point is spanned by a family
of degenerate and non-commuting vector fields $\left\{ X_{i},i\in\mathcal{R},1<\vert\mathcal{R}\vert<\infty\right\} $,
where $\vert\mathcal{R}\vert$ is the cardinality of the index set
$\mathcal{R}.$ These inequalities, when satisfied, give us information
about the spectra of the associated generators of the form 
\begin{equation}
\mathcal{L}=\sum_{i\in R}X_{i}^{2}\label{eq:2}
\end{equation}
where $\vert\mathcal{R}\vert$ is strictly less than the dimension
of the space. (See also \cite{key-17} and references therein) 

Thus, by Hörmander's result in \cite{key-4}, the sub-Laplacian (\ref{eq:2})
is hypoelliptic; in other words, every distributional solution to
$\mathcal{L}u=f$ is of class $C^{\infty}$ whenever f is of class
$C^{\infty}.$

We point out that, according to \cite{key-10}, if we have a uniqueness
of solution in the space of square integrable functions for the Cauchy
problem
\[
\left\{ \begin{array}{ll}
\frac{\partial u}{\partial t}=\mathcal{L}u\\
u\vert_{t=0}=f,
\end{array}\right.
\]
then the solution of the heat equation will be given by ${\displaystyle u=P_{t}f.}$ 

In the setting of step-two nilpotent Lie groups, since the Laplacian
is of Hörmander type and has some degeneracy, D. Bakry and M. Emery\textquoteright s
Curvature-Dimension condition in \cite{key-1} will no longer hold
true. In \cite{key-6}, a method of studying coercive inequalities
on general metric spaces that does not require a bound on the curvature
of space was developed. Working on a general metric space equipped
with non-commuting vector fields $\{X_{1},\ldots,X_{n}\},$ their
method is based on U-bounds, which are inequalities of the form:

\[
\int f^{q}\mathcal{U}\left(d\right)d\mu\leq C\int\vert\triangledown f\vert^{q}d\mu+D\int f^{q}d\mu
\]
where ${\displaystyle d\mu=\frac{e^{-U\left(d\right)}}{Z}d\lambda}$
is a probability measure, $U(d)$ and $\mathcal{U}\left(d\right)$
are functions having a suitable growth at infinity, $\lambda$ is
a natural measure like the Lebesgue measure for instance (which is
the Haar measure for nilpotent Lie groups), $d$ is a metric related
to the gradient ${\displaystyle \triangledown=\left(X_{1},\ldots,X_{n}\right),}$
and $q\in\left(1,\infty\right).$ 

It is worth mentioning that in the setting of nilpotent Lie groups,
heat kernel estimates were studied to get a variety of coercive inequalities
(\cite{key-19,key-20,key-21,key-32,key-23,key-25,key-26,key-27,key-28,key-24}).
In our setting, we study coercive inequalities involving sub-gradients
and probability measures on the group which is a difficult and much
less explored subject. An approach, pioneered in \cite{key-6}, was
used by J. Inglis to get Poincaré inequality in the setting of the
Heisenberg-type group with measure as a function of Kaplan distance
\cite{key-7} and by M. Chatzakou et al. to get Poincaré inequality
in the setting of the Engel-type group with a measure as a function
of some homegenous norm \cite{key-30}. 
\end{onehalfspace}

In section 2 we define the step-two Carnot group, and introduce $N,$
the homogeneous norm we are working with, that is of the form of the
Kaplan norm in the Heisenberg-type group. Section 3 contains the main
theorem, which is a proof of a U-Bound of the form
\[
\int\frac{g'\left(N\right)}{N^{2}}\vert f\vert^{q}d\mu\leq C\int\vert\triangledown f\vert^{q}d\mu+D\int\vert f\vert^{q}d\mu,
\]
where $g(N)$ satisfies some growth conditions. In section 4, we apply
this U-bound together with some results of \cite{key-6} to get the
q-Poincaré inequality with $q\geq2$ for the measures ${\displaystyle d\mu=\frac{e^{-g(N)}}{Z}d\lambda.}$
This generalises the result by J. Inglis \cite{key-7} who, in the
setting of the Heisenberg-type group, proved the q-Poincaré inequality
for the measure ${\displaystyle d\mu=\frac{e^{-\alpha\tilde{N}^{p}}}{Z}d\lambda,}$
where $p\geq2,$ $\alpha>0,$ $q$ is the finite index conjugate to
$p,$ and with $\tilde{N}$ the Kaplan norm. In section 5, we extend
J. Inglis et al.'s Theorem 2.1 \cite{key-8} who proved a Log$^{\beta}-$Sobolev,
$\beta\in(0,1)$ inequality in the context of the Heisenberg group.
Recall that for density defined with a smooth homogenous norm, $\beta=1$
is not allowed (\cite{key-6}). We extend the corresponding results
to a $\phi-$Logarithmic Sobolev inequality, where $\phi$ is concave,
on step-two Carnot groups. Finally, we utilise the U-Bound to get
to a Log$^{\beta}-$Sobolev inequality for $d\mu=\frac{e^{-\alpha N^{p}}}{Z}d\lambda,$
where $p\geq4,$ $q\geq2,$ and $0<\beta\leq{\displaystyle \frac{p-3}{p}},$
indicating also no-go zone of parameters where the corresponding inequality
fails. 

\section{Setup}

Carnot groups are geodesic metric spaces that appear in many mathematical
contexts like harmonic analysis in the study of hypoelliptic differential
operators (\cite{key-1,key-10}) and in geometric measure theory (see
extensive reference list in the survey paper \cite{key-2}). 

We will be working in the setting of $\mathbb{G},$ a step-two Carnot
group, i.e. a group isomorphic to ${\displaystyle \mathbb{R}^{n+m}}$
with the group law
\[
\left(x,z\right)\circ\left(x',z'\right)=\left(x_{i}+x'_{i},~z_{j}+z'_{j}+\frac{1}{2}<\Lambda^{\left(j\right)}x,x'>\right)_{i=1,..,n;j=1,..,m}
\]
for $x,x'\in\mathbb{R}^{n},z,z'\in\mathbb{R}^{m}$, where the matrices
$\Lambda^{\left(j\right)}$ are $n\times n$ skew-symmetric and linearly
independent and $<.,.>$ stands for the inner producr on $\mathbb{R}^{n}$.
One can verify that $(\mathbb{R}^{n+m},\circ)$ is a Lie group whose
identity is the origin and where the inverse is given by $(x,z)^{-1}=-(x,z).$

The dilation
\[
\delta_{\lambda}:{\displaystyle \mathbb{R}^{n+m}}\rightarrow{\displaystyle \mathbb{R}^{n+m}},\;\;\;\;\;\;\delta_{\lambda}(x,z)=(\lambda x,\lambda^{2}z)
\]

is an automorphism of $({\displaystyle \mathbb{R}^{n+m}},\circ)$
for any $\lambda>0.$ Then, $\mathbb{G}=(\mathbb{R}^{n+m},\circ,\delta_{\lambda})$
is a homogeneous Lie group.

The Jacobian matrix at $(0,0)$ of the left translation $\tau_{(x,z)}$
i.e the map
\[
{\displaystyle \mathbb{G}\ni(x',z')\rightarrow\tau_{(x,z)}((x',z')):=(x,z)\circ(x',z')\in\mathbb{G}}
\]
for fixed $(x,z)\in\mathbb{G}$) takes the following form
\begin{align*}
\mathcal{J}_{\tau_{(x,z)}}(0,0) & =\left(\begin{array}{ccc|c}
 & \mathbb{I}_{n} &  & 0_{n\times m}\\
\hline {\displaystyle \frac{1}{2}\sum_{l=1}^{n}\Lambda_{1l}^{\left(1\right)}x_{l}} & \cdots & \frac{1}{2}\sum_{l=1}^{n}\Lambda_{nl}^{\left(1\right)}x_{l}\\
\vdots & \cdots & \vdots & \mathbb{I}_{m}\\
\frac{1}{2}\sum_{l=1}^{n}\Lambda_{1l}^{\left(m\right)}x_{l} & \cdots & \frac{1}{2}\sum_{l=1}^{n}\Lambda_{nl}^{\left(m\right)}x_{l}
\end{array}\right).
\end{align*}

Then, the Jacobian basis of $\mathfrak{g}$, the Lie algebra of $\mathbb{G},$
is given by

\[
X_{i}=\frac{\partial}{\partial x_{i}}+\frac{1}{2}\sum_{k=1}^{m}\sum_{l=1}^{n}\Lambda_{il}^{\left(k\right)}x_{l}\frac{\partial}{\partial z_{k}}\qquad\textrm{ and }\qquad Z_{j}=\frac{\partial}{\partial z_{j}},
\]
for $i\in\left\{ 1,\ldots,n\right\} $ and $j\in\left\{ 1,\ldots,m\right\} .$

Let $\nabla\equiv(X_{i})_{i=1,..,n}$ and $\ensuremath{}\Delta\equiv{\displaystyle {\displaystyle }\sum_{i=1,..,n}X_{i}^{2}}$
denote the associated sub-gradient and sub-Laplacian, respectively.
We consider the following smooth homogeneous norm on $\ensuremath{\mathbb{G}}$
\[
N\equiv\left(|x|^{4}+a|z|^{2}\right)^{\frac{1}{4}}
\]
with $a\in(0,\infty).$

The motivation behind choosing such a norm is that in the setting
of the Heisenberg-type groups (where we assume in addition that $\Lambda^{\left(j\right)}$
are orthogonal matrices and that $\Lambda^{\left(i\right)}\Lambda^{\left(j\right)}=-\Lambda^{\left(j\right)}\Lambda^{\left(i\right)}$
for every $i,j\in\{1,..,m\}$ with $i\neq j$), $N\equiv\left(|x|^{4}+16|z|^{2}\right)^{\frac{1}{4}}$
is the Kaplan norm which arises from the fundamental solution of the
sub-Laplacian. In other words, $\triangle N^{2-n-2m}=0$ in $\mathbb{G}\backslash\{0\}.$
Recall that J. Inglis, in \cite{key-7}, proved the q-Poincaré inequality
in the setting of the Heisenberg-type group for the measure ${\displaystyle d\mu=\frac{e^{-\alpha N^{p}}}{Z}d\lambda,}$
where $p\geq2,$ $\alpha>0,$ $q$ is the finite index conjugate to
$p$, and $N\equiv\left(|x|^{4}+16|z|^{2}\right)^{\frac{1}{4}}$ .
We extend, giving a simpler proof, the result of J. Inglis by obtaining
a q-Poincaré inequality in the setting of step-two Carnot groups for
the measures ${\displaystyle d\mu=\frac{e^{-g(N)}}{Z}d\lambda,}$
with ${\displaystyle \frac{g'(N)}{N^{2}}}$ an increasing function,
$q\geq2,$ and where $N\equiv\left(|x|^{4}+a|z|^{2}\right)^{\frac{1}{4}}.$ 

Our first key result this paper is obtaining the following U-Bound
(section 3):
\begin{equation}
\int\frac{g'\left(N\right)}{N^{2}}\vert f\vert^{q}d\mu\leq C\int\vert\triangledown f\vert^{q}d\mu+D\int\vert f\vert^{q}d\mu,\label{eq:u}
\end{equation}

under certain growth conditions for $g(N).$ This U-bound is a useful
tool to get a q-Poincaré inequality (section 4) and a Log$^{\beta}$-Sobolev
inequality (section 5) for $q\geq2.$ We expect that this U-bound
can be used to extend those coercive inequalities to (nonproduct)
measures in an infinite dimensional setting. \cite{key-67}

\section{U-Bound}
\begin{thm}
Let $N=\left(|x|^{4}+a|z|^{2}\right)^{\frac{1}{4}}$ with $a\in(0,\infty),$
and let $g:\left[0,\infty\right)\rightarrow\left[0,\infty\right)$
be a differentiable increasing function such that $g''(N)\leq g'(N)^{3}N^{3}$
on $\{N\geq1\}$. Let ${\displaystyle d\mu=\frac{e^{-g\left(N\right)}}{Z}d\lambda}$
be a probability measure, where $Z$ is the normalization constant.
Then, given $q\ensuremath{\geq2},$
\[
\int\frac{g'\left(N\right)}{N^{2}}\vert f\vert^{q}d\mu\leq C\int\vert\triangledown f\vert^{q}d\mu+D\int\vert f\vert^{q}d\mu
\]
holds for all locally Lipschitz functions $f,$ supported outside
the unit ball $\left\{ N<1\right\} ,$ with $C$ and $D$ positive
constants independent of $f.$
\end{thm}

The proof of Theorem 1 uses the following properties of a smooth norm
$N$ proven in the Appendix.
\begin{lem}
\begin{singlespace}
There exist constants $A,C\in(0,\infty)$
\begin{equation}
A\frac{|x|^{2}}{N^{2}}\leq\left|\nabla N\right|^{2}\leq C\frac{|x|^{2}}{N^{2}}\label{eq:1-1}
\end{equation}
and there exists a constant $B\in(0,\infty)$ such that
\begin{equation}
|\Delta N|\leq B\frac{|x|^{2}}{N^{3}}\label{eq:2-1}
\end{equation}
and
\begin{equation}
\frac{x}{|x|}\cdot\nabla N=\frac{|x|^{3}}{N^{3}}.\label{eq:3}
\end{equation}
\end{singlespace}
\end{lem}

The other main tools we use are Hardy's inequality (see \cite{key-69-1}
and references therein) and the Coarea formula (page 468 of \cite{key-4}). 
\begin{proof}[Proof of Theorem 1]
First, we prove the result for $q=2:$

We note that using integration by parts, one gets
\[
\int\left(\triangledown N\right)\cdot\left(\triangledown f\right)e^{-g\left(N\right)}d\lambda=-\int\triangledown\left(\triangledown Ne^{-g\left(N\right)}\right)fd\lambda=-\int\Delta Nfe^{-g\left(N\right)}d\lambda+\int\vert\triangledown N\vert^{2}fg^{'}\left(N\right)e^{-g\left(N\right)}d\lambda.
\]

Netx, using (\ref{eq:1-1}) and (\ref{eq:2-1}),
\[
\int\left(\triangledown N\right)\cdot\left(\triangledown f\right)e^{-g\left(N\right)}d\lambda\geq-B\int\frac{|x|^{2}}{N^{3}}fe^{-g\left(N\right)}d\lambda+A\int\frac{|x|^{2}}{N^{2}}fg^{'}\left(N\right)e^{-g\left(N\right)}d\lambda.
\]

Replacing $f$ by ${\displaystyle \frac{f^{2}}{\vert x\vert^{2}}}$
:
\begin{equation}
\int\left(\triangledown N\right)\cdot\left(\triangledown\left(\frac{f^{2}}{|x|^{2}}\right)\right)e^{-g\left(N\right)}d\lambda\geq\int f^{2}\left(\frac{Ag'(N)}{N^{2}}-\frac{B}{N^{3}}\right)e^{-g\left(N\right)}d\lambda.\label{eq:4}
\end{equation}

As for the left-hand side of (\ref{eq:4}),
\[
\begin{array}{cl}
{\displaystyle \int\left(\triangledown N\right)\cdot\left(\triangledown\left(\frac{f^{2}}{\vert x\vert^{2}}\right)\right)e^{-g\left(N\right)}d\lambda} & {\displaystyle =\int\left(\triangledown N\right)\cdot\left[2f\frac{\triangledown f}{\vert x\vert^{2}}-\frac{2f^{2}\triangledown\vert x\vert}{\vert x\vert^{3}}\right]e^{-g\left(N\right)}d\lambda}\\
\\
 & {\displaystyle =\int\frac{2f}{\vert x\vert^{2}}\triangledown N\cdot\triangledown fe^{-g\left(N\right)}d\lambda-2\int f^{2}\frac{\triangledown N\cdot x}{\vert x\vert^{4}}e^{-g\left(N\right)}d\lambda}
\end{array}
\]

Using the calculation of $\triangledown N\cdot x,$ from (\ref{eq:3}),
we get:
\[
\begin{array}{cl}
\;\;\;\;\;\;\;\;\;\;\;\;\;\;\;\;\;\;\; & {\displaystyle =\int\frac{2f}{\vert x\vert^{2}}\triangledown N\cdot\triangledown fe^{-g\left(N\right)}d\lambda-2\int\frac{f^{2}}{N^{3}}e^{-g\left(N\right)}d\lambda}\\
\\
 & {\displaystyle \leq\int\frac{2f}{\vert x\vert^{2}}\triangledown N\cdot\triangledown fe^{-g\left(N\right)}d\lambda.}
\end{array}
\]

Combining with (\ref{eq:4}),

\[
\begin{array}{cl}
{\displaystyle \int f^{2}\left(\frac{Ag'(N)}{N^{2}}-\frac{B}{N^{3}}\right)e^{-g\left(N\right)}d\lambda} & {\displaystyle \leq\int\left(\triangledown N\right)\cdot\left(\triangledown\left(\frac{f^{2}}{\vert x\vert^{2}}\right)\right)e^{-g\left(N\right)}d\lambda}\\
\\
 & {\displaystyle \leq2\left|\int\frac{f}{\vert x\vert^{2}}\triangledown N\cdot\triangledown fe^{-g\left(N\right)}d\lambda\right|}\\
\\
 & {\displaystyle \leq2\int\frac{f}{\vert x\vert^{2}}|\triangledown N||\triangledown f|e^{-g\left(N\right)}d\lambda}
\end{array}
\]

using (\ref{eq:1-1}),
\[
\leq2\sqrt{C}\int\frac{\vert f\vert}{N\vert x\vert}\vert\triangledown f\vert e^{-g\left(N\right)}d\lambda
\]

Let $E=\{(x,z):|x|\geq\frac{1}{\sqrt{g'(N)}}\}$ and $F=\{(x,z):|x|<\frac{1}{\sqrt{g'(N)}}\}.$ 

Applying Cauchy\textquoteright s inequality with $\epsilon:$ ${\displaystyle ab\leq\frac{\epsilon a^{2}}{2}+\frac{b^{2}}{2\epsilon}}$
with ${\displaystyle a=\frac{\vert f\vert}{N\vert x\vert}e^{-\frac{g\left(N\right)}{2}}}$
and ${\displaystyle b=\sqrt{C}\vert\triangledown f\vert e^{-\frac{g\left(N\right)}{2}},}$ 

to obtain
\begin{equation}
\begin{array}{cl}
{\displaystyle \int f^{2}\left(\frac{Ag'(N)}{N^{2}}-\frac{B}{N^{3}}\right)e^{-g\left(N\right)}d\lambda} & {\displaystyle \leq\epsilon\int\frac{\vert f\vert^{2}}{N^{2}\vert x\vert^{2}}e^{-g\left(N\right)}d\lambda+\frac{C}{\epsilon}\int\vert\triangledown f\vert^{2}e^{-g\left(N\right)}d\lambda}\\
\\
 & {\displaystyle =\epsilon\int_{F}\frac{\vert f\vert^{2}}{N^{2}\vert x\vert^{2}}e^{-g\left(N\right)}d\lambda+\epsilon\int_{E}\frac{\vert f\vert^{2}}{N^{2}\vert x\vert^{2}}e^{-g\left(N\right)}d\lambda+\frac{C}{\epsilon}\int\vert\triangledown f\vert^{2}e^{-g\left(N\right)}d\lambda}\\
\\
 & {\displaystyle \leq\epsilon\int_{F}\frac{\vert fe^{\frac{-g(N)}{2}}\vert^{2}}{N^{2}|x|^{2}}d\lambda+\epsilon\int\frac{g'(N)\vert f\vert^{2}}{N^{2}}e^{-g\left(N\right)}d\lambda+\frac{C}{\epsilon}\int\vert\triangledown f\vert^{2}e^{-g\left(N\right)}d\lambda,}
\end{array}\label{eq:19}
\end{equation}
where (\ref{eq:19}) is the consequence of $E=\{(x,z):|x|\geq\frac{1}{\sqrt{g'(N)}}\}.$\\

The aim now is to estimate the first term on the right-hand side of
(\ref{eq:19}). Consider $F_{r}=\left\{ \left|x\right|\sqrt{g'\left(N\right)}<r\right\} ,$
where $1<r<2.$ Integrating by parts:

\begin{doublespace}
\[
\epsilon\int_{F_{r}}\frac{\vert fe^{\frac{-g(N)}{2}}\vert^{2}}{N^{2}|x|^{2}}d\lambda=\frac{\epsilon}{n-2}\int_{F_{r}}\frac{\vert fe^{\frac{-g(N)}{2}}\vert^{2}}{N^{2}}\nabla\left(\frac{x}{|x|^{2}}\right)d\lambda
\]
\[
\begin{array}{c}
{\displaystyle =-\frac{\epsilon}{n-2}\int_{F_{r}}\nabla\left(\left\vert \frac{fe^{\frac{-g(N)}{2}}}{N}\right\vert ^{2}\right)\cdot\frac{x}{|x|^{2}}d\lambda+\frac{\epsilon}{n-2}\int_{\partial F_{r}}\frac{f^{2}e^{-g(N)}}{N^{2}|x|^{2}}\sum_{j=1}^{n}\frac{x_{j}<X_{j}I,\nabla_{euc}\left(|x|\sqrt{g'(N)}\right)>}{\left|\nabla_{euc}\left(|x|\sqrt{g'(N)}\right)\right|}dH^{n+m-1}}\\
\\
{\displaystyle =-\frac{2\epsilon}{n-2}\int_{F_{r}}\frac{fe^{\frac{-g(N)}{2}}}{N}\nabla\left(\frac{fe^{\frac{-g(N)}{2}}}{N}\right)\cdot\frac{x}{|x|^{2}}d\lambda}\\
{\displaystyle +\frac{\epsilon}{n-2}\int_{\partial F_{r}}\frac{f^{2}e^{-g(N)}}{N^{2}|x|^{2}}\sum_{j=1}^{n}\frac{x_{j}<X_{j}I,\nabla_{euc}\left(|x|\sqrt{g'(N)}\right)>}{\left|\nabla_{euc}\left(|x|\sqrt{g'(N)}\right)\right|}dH^{n+m-1}}\\
\\
{\displaystyle \leq\frac{\epsilon}{2}\int_{F_{r}}\frac{\vert fe^{\frac{-g(N)}{2}}\vert^{2}}{N^{2}|x|^{2}}d\lambda+\frac{2\epsilon}{(n-2)^{2}}\int_{F_{r}}\left|\nabla\left(\frac{fe^{\frac{-g(N)}{2}}}{N}\right)\right|^{2}d\lambda}\\
{\displaystyle +\frac{\epsilon}{n-2}\int_{\partial F_{r}}\frac{f^{2}e^{-g(N)}}{N^{2}|x|^{2}}\sum_{j=1}^{n}\frac{x_{j}<X_{j}I,\nabla_{euc}\left(|x|\sqrt{g'(N)}\right)>}{\left|\nabla_{euc}\left(|x|\sqrt{g'(N)}\right)\right|}dH^{n+m-1},}
\end{array}
\]
where in the last step we used Cauchy's inequality. Subtracting on
both sides of the last inequality by ${\displaystyle \frac{\epsilon}{2}\int_{F_{r}}\frac{\vert fe^{\frac{-g(N)}{2}}\vert^{2}}{N^{2}|x|^{2}}d\lambda,}$
and using the fact that $1<r<2,$ we get:
\[
\epsilon\int_{F_{1}}\frac{\vert fe^{\frac{-g(N)}{2}}\vert^{2}}{N^{2}|x|^{2}}d\lambda\leq\epsilon\int_{F_{r}}\frac{\vert fe^{\frac{-g(N)}{2}}\vert^{2}}{N^{2}|x|^{2}}d\lambda
\]
\[
\leq\frac{4\epsilon}{(n-2)^{2}}\int_{F_{r}}\left|\nabla\left(\frac{fe^{\frac{-g(N)}{2}}}{N}\right)\right|^{2}d\lambda+\frac{2\epsilon}{n-2}\int_{\partial F_{r}}\frac{f^{2}e^{-g(N)}}{N^{2}|x|^{2}}\sum_{j=1}^{n}\frac{x_{j}<X_{j}I,\nabla_{euc}\left(|x|\sqrt{g'(N)}\right)>}{\left|\nabla_{euc}\left(|x|\sqrt{g'(N)}\right)\right|}dH^{n+m-1}
\]
\[
\leq\frac{4\epsilon}{(n-2)^{2}}\int_{F_{2}}\left|\nabla\left(\frac{fe^{\frac{-g(N)}{2}}}{N}\right)\right|^{2}d\lambda+\frac{2\epsilon}{n-2}\int_{\partial F_{r}}\frac{f^{2}e^{-g(N)}}{N^{2}|x|^{2}}\sum_{j=1}^{n}\frac{x_{j}<X_{j}I,\nabla_{euc}\left(|x|\sqrt{g'(N)}\right)>}{\left|\nabla_{euc}\left(|x|\sqrt{g'(N)}\right)\right|}dH^{n+m-1}
\]
\end{doublespace}

Integrating both sides of the inequality from $r=1$ to $r=2,$ we
get:

\[
\epsilon\int_{1}^{2}\int_{F_{1}}\frac{\vert fe^{\frac{-g(N)}{2}}\vert^{2}}{N^{2}|x|^{2}}d\lambda dr\leq\frac{4\epsilon}{(n-2)^{2}}\int_{1}^{2}\int_{F_{2}}\left|\nabla\left(\frac{fe^{\frac{-g(N)}{2}}}{N}\right)\right|^{2}d\lambda dr
\]

\[
+\frac{2\epsilon}{n-2}\int_{1}^{2}\int_{\partial F_{r}}\frac{f^{2}e^{-g(N)}}{N^{2}|x|^{2}}\sum_{j=1}^{n}\frac{x_{j}<X_{j}I,\nabla_{euc}\left(|x|\sqrt{g'(N)}\right)>}{\left|\nabla_{euc}\left(|x|\sqrt{g'(N)}\right)\right|}dH^{n+m-1}dr
\]
To recover the full measure in the boundary term, we use the Coarea
formula: 
\begin{equation}
\epsilon\int_{F_{1}}\frac{\vert fe^{\frac{-g(N)}{2}}\vert^{2}}{N^{2}|x|^{2}}d\lambda\leq\frac{4\epsilon}{(n-2)^{2}}\int_{F_{2}}\left|\nabla\left(\frac{fe^{\frac{-g(N)}{2}}}{N}\right)\right|^{2}d\lambda\label{eq:20}
\end{equation}
\[
+\frac{2\epsilon}{n-2}\int_{\{1<|x|\sqrt{g'(N)}<2\}}\frac{f^{2}e^{-g(N)}}{N^{2}|x|^{2}}\sum_{j=1}^{n}x_{j}<X_{j}I,\nabla_{euc}\left(|x|\sqrt{g'(N)}\right)>d\lambda
\]

It remains to compute the right hand side of (\ref{eq:20}). The first
term,
\[
\begin{array}{cc}
A & {\displaystyle =\frac{4\epsilon}{(n-2)^{2}}\int_{F_{2}}\left|\nabla\left(\frac{fe^{\frac{-g(N)}{2}}}{N}\right)\right|^{2}d\lambda=\frac{4\epsilon}{(n-2)^{2}}\int_{F_{2}}\left|\nabla f\frac{e^{-\frac{g(N)}{2}}}{N}-f\frac{g'(N)}{2}\frac{\nabla Ne^{-\frac{g(N)}{2}}}{N}-\frac{f\nabla Ne^{-\frac{g(N)}{2}}}{N^{2}}\right|^{2}d\lambda}\\
\\
 & {\displaystyle \leq\frac{16\epsilon}{(n-2)^{2}}\int_{F_{2}}\frac{|\nabla f|^{2}}{N^{2}}e^{-g(N)}d\lambda+\frac{4\epsilon}{(n-2)^{2}}\int_{F_{2}}f^{2}g'(N)^{2}\frac{|\nabla N|^{2}}{N^{2}}e^{-g(N)}d\lambda+\frac{16\epsilon}{(n-2)^{2}}\int_{F_{2}}f^{2}\frac{|\nabla N|^{2}}{N^{4}}e^{-g(N)}d\lambda}
\end{array}
\]
Using (\ref{eq:1-1}) and taking into consideration that $N>1$ and
on $F_{2},$ ${\displaystyle |\nabla N|^{2}\leq\frac{C|x|^{2}}{N^{2}}\leq\frac{4C}{N^{2}g'(N)},}$
\\
\[
A\leq\tilde{C}\int_{F_{2}}|\nabla f|^{2}e^{-g(N)}d\lambda+\tilde{D}\int_{F_{2}}f^{2}e^{-g(N)}d\lambda+\frac{16\epsilon C}{(n-2)^{2}}\int_{F_{2}}f^{2}\frac{g'(N)}{N^{4}}e^{-g(N)}d\lambda
\]

We do not worry about the third term in this inequality since it is
dominated by $\ensuremath{\int f^{2}\frac{g'(N)}{N^{2}}e^{-g(N)}d\lambda}$
for $N>1.$ For the second term of (\ref{eq:20}), 
\[
B=\frac{2\epsilon}{n-2}\int_{\{1<|x|\sqrt{g'(N)}<2\}}\frac{f^{2}e^{-g(N)}}{N^{2}|x|^{2}}\sum_{j=1}^{n}x_{j}<X_{j}I,\nabla_{euc}\left(|x|\sqrt{g'(N)}\right)>d\lambda.
\]

For $e_{i}$ the standard Euclidean basis on $\mathbb{R}^{n+m},$

\[
X_{j}I\cdot e_{i}=\begin{cases}
0\;\;\;\;\;\;\;\;\;\;\;\;\;\;\;\;\;\;\;\;\;\;\; & for\;\;\;i\neq j\;\;\;and\;\;\;i\leq n\\
1\;\;\;\;\;\;\;\;\;\;\;\;\;\;\;\;\;\;\;\;\;\;\; & for\;\;\;i=j\;\;\;and\;\;\;i\leq n\\
\frac{1}{2}\sum_{l=1}^{n}\Lambda_{jl}^{\left(i\right)}x_{l}\;\;\;\;\;\;\; & for\;\;\;n+1\leq i\leq n+m.
\end{cases}
\]

\[
\nabla_{euc}\left(|x|\sqrt{g'(N)}\right)\cdot e_{i}=\begin{cases}
\frac{x_{i}\sqrt{g'(N)}}{|x|}+\frac{|x|^{3}g''(N)x_{i}}{2\sqrt{g'(N)}N^{3}}\;\;\;\;\;\;\;\;\;\;\;\;\;\;\;\;\;\;for\;\;\;i=j\;\;\;and\;\;\;i\leq n\\
\frac{a|x|g''(N)z_{i}}{4\sqrt{g'(N)}N^{3}}\;\;\;\;\;\;\;\;\;\;\;\;\;\;\;\;\;\;\;\;\;\;\;\;\;\;\;\;\;\;\;\;\;\;for\;\;\;n+1\leq i\leq n+m
\end{cases}
\]

Taking the dot product and summing,
\[
\begin{array}{cl}
{\displaystyle \sum_{j=1}^{n}x_{j}<X_{j}I,\nabla_{euc}\left(|x|\sqrt{g'(N)}\right)>} & {\displaystyle =|x|\sqrt{g'(N)}+\frac{|x|^{5}g''(N)}{2\sqrt{g'(N)}N^{3}}+\sum_{j=1}^{n}x_{j}\sum_{i=n+1}^{n+m}\left(\frac{a|x|g''(N)}{8\sqrt{g'(N)}N^{3}}\right)z_{i}\sum_{l=1}^{n}\Lambda_{jl}^{(i)}x_{l}}\\
\\
 & {\displaystyle =|x|\sqrt{g'(N)}+\frac{|x|^{5}g''(N)}{2\sqrt{g'(N)}N^{3}}+\left(\frac{a|x|g''(N)}{8\sqrt{g'(N)}N^{3}}\right)\sum_{i=n+1}^{n+m}z_{i}\sum_{j=1}^{n}\sum_{l=1}^{n}\Lambda_{jl}^{(i)}x_{l}x_{j}}\\
\\
 & {\displaystyle =|x|\sqrt{g'(N)}+\frac{|x|^{5}g''(N)}{2\sqrt{g'(N)}N^{3}},}
\end{array}
\]

where ${\displaystyle \sum_{j=1}^{n}\sum_{l=1}^{n}\Lambda_{jl}^{(i)}x_{l}x_{j}=0}$
since ${\displaystyle \Lambda_{jl}^{(i)}}$ is skew symmetric.\\

Therefore, replacing, 
\[
\begin{array}{cl}
B & {\displaystyle =\frac{2\epsilon}{n-2}\int_{\{1<|x|\sqrt{g'(N)}<2\}}\frac{f^{2}e^{-g(N)}}{N^{2}|x|^{2}}\sum_{j=1}^{n}x_{j}<X_{j}I,\nabla_{euc}\left(|x|\sqrt{g'(N)}\right)>d\lambda}\\
\\
 & {\displaystyle =\frac{2\epsilon}{n-2}\int_{\{1<|x|\sqrt{g'(N)}<2\}}\frac{f^{2}\sqrt{g'(N)}}{N^{2}|x|}e^{-g(N)}d\lambda+\frac{2\epsilon}{n-2}\int_{\{1<|x|\sqrt{g'(N)}<2\}}\frac{f^{2}|x|^{3}g''(N)}{2N^{5}\sqrt{g'(N)}}e^{-g(N)}d\lambda.}
\end{array}
\]
Using the fact that we are integrating over $\{1<|x|\sqrt{g'(N)}<2\},$
\[
B\leq\frac{2\epsilon}{n-2}\int_{\{1<|x|\sqrt{g'(N)}<2\}}\frac{f^{2}g'(N)}{N^{2}}e^{-g(N)}d\lambda+\frac{8\epsilon}{n-2}\int_{\{1<|x|\sqrt{g'(N)}<2\}}\frac{f^{2}g''(N)}{N^{5}g'(N)^{2}}e^{-g(N)}d\lambda.
\]

Using the condition of the theorem that $g''(N)\leq g'(N)^{3}N^{3},$
we get
\[
B\leq\frac{10\epsilon}{n-2}\int_{\{1<|x|\sqrt{g'(N)}<2\}}\frac{f^{2}g'(N)}{N^{2}}e^{-g(N)}d\lambda.
\]

Inserting bounds on $A$ and $B$ in (\ref{eq:20}), we get:
\[
\epsilon\int_{F_{1}}\frac{|fe^{\frac{-g(N)}{2}}|^{2}}{N^{2}|x|^{2}}d\lambda\leq\tilde{C}\int_{F_{2}}|\nabla f|^{2}e^{-g(N)}d\lambda+\tilde{D}\int_{F_{2}}f^{2}e^{-g(N)}d\lambda
\]
\[
+\frac{16\epsilon C}{(n-2)^{2}}\int_{F_{2}}f^{2}\frac{g'(N)}{N^{4}}e^{-g(N)}d\lambda+\frac{10\epsilon}{n-2}\int_{\{1<|x|\sqrt{g'(N)}<2\}}\frac{f^{2}g'(N)}{N^{2}}e^{-g(N)}d\lambda.
\]

Using this last bound to estimate (\ref{eq:19}), we get:
\[
\int f^{2}\left(\frac{Ag'(N)}{N^{2}}-\frac{B}{N^{3}}\right)e^{-g\left(N\right)}d\lambda\leq
\]
\[
\tilde{C}\int|\nabla f|^{2}e^{-g(N)}d\lambda+\tilde{D}\int f^{2}e^{-g(N)}d\lambda+\frac{16\epsilon C}{(n-2)^{2}}\int f^{2}\frac{g'(N)}{N^{4}}e^{-g(N)}d\lambda+\left(\frac{10\epsilon}{n-2}+\epsilon\right)\int\frac{f^{2}g'(N)}{N^{2}}e^{-g(N)}d\lambda.
\]
On $\{N>1\},$ ${\displaystyle \int f^{2}\left(\frac{B}{N^{3}}\right)e^{-g\left(N\right)}d\lambda}$
and ${\displaystyle \frac{16\epsilon C}{(n-2)^{2}}\int f^{2}\frac{g'(N)}{N^{4}}e^{-g(N)}d\lambda}$
are of lower order. So, choosing ${\displaystyle \left(\frac{10\epsilon}{n-2}+\epsilon\right)<A}$
we get
\[
\int f^{2}\left(\frac{g^{'}\left(N\right)}{N^{2}}\right)e^{-g\left(N\right)}d\lambda\leq C\int|\nabla f|^{2}e^{-g(N)}d\lambda+D\int f^{2}e^{-g(N)}d\lambda.
\]

Secondly, for $q>2,$ replacing $|f|$ by ${\displaystyle \vert f\vert^{\frac{q}{2}},}$we
get:
\begin{equation}
\int\frac{g'\left(N\right)}{N^{2}}\vert f\vert^{q}d\mu\leq C\int\left|\triangledown\vert f\vert^{\frac{q}{2}}\right|^{2}d\mu+D\int\vert f\vert^{q}d\mu.\label{eq:21}
\end{equation}

Calculating,
\[
\begin{array}{cl}
{\displaystyle \int\left|\triangledown\vert f\vert^{\frac{q}{2}}\right|^{2}d\mu} & {\displaystyle =\int\left|\frac{q}{2}\vert f\vert^{\frac{q-2}{2}}\left(sgn\left(f\right)\right)\triangledown f\right|^{2}d\mu}\\
\\
 & {\displaystyle \leq\int\frac{q^{2}}{4}\vert f\vert^{q-2}\vert\triangledown f\vert^{2}d\mu.}
\end{array}
\]
\textbf{Remark: }We note that at this point we get the inequality
which implies the necessary and sufficient condition for exponential
decay in $\mathbb{L}_{p}$ as described in \cite{key-16}.

Using Hölder\textquoteright s inequality,
\begin{equation}
\begin{array}{cl}
{\displaystyle \int\left|\triangledown\vert f\vert^{\frac{q}{2}}\right|^{2}d\mu} & {\displaystyle \leq\int\frac{q^{2}}{4}\vert f\vert^{q-2}\vert\triangledown f\vert^{2}d\mu~\leq\frac{q^{2}}{4}\left(\int\vert f\vert^{q}d\mu\right)^{\frac{q-2}{q}}\left(\int\vert\triangledown f\vert^{q}d\mu\right)^{\frac{2}{q}}}\\
\\
 & {\displaystyle \leq\frac{q\left(q-2\right)}{4}\int\vert f\vert^{q}d\mu+\frac{q}{2}\int\vert\triangledown f\vert^{q}d\mu.}
\end{array}\label{eq:22}
\end{equation}

Where the last inequality uses $ab{\displaystyle \leq\frac{a^{p^{'}}}{p^{'}}+\frac{b^{q'}}{q'},}$
with ${\displaystyle a=\left(\int\vert f\vert^{q}d\mu\right)^{\frac{q-2}{q}},}$
${\displaystyle b=\left(\int\vert\triangledown f\vert^{q}d\mu\right)^{\frac{2}{q}},}$
and $p^{'}$ and $q'$ are conjugates. 

Choosing ${\displaystyle p^{'}=\frac{q}{q-2}},$ we obtain ${\displaystyle \frac{1}{q'}=1-\frac{q-2}{q}}$
, so ${\displaystyle q^{'}=\frac{q}{2}}.$ Using the inequalities
(\ref{eq:21}) and (\ref{eq:22}), we get,
\[
\begin{array}{cl}
{\displaystyle \int\frac{g'\left(N\right)}{N^{2}}\vert f\vert^{q}d\mu} & {\displaystyle \leq C\int\left|\triangledown\vert f\vert^{\frac{q}{2}}\right|^{2}d\mu+D\int\vert f\vert^{q}d\mu}\\
\\
 & {\displaystyle \leq C'\int\vert\triangledown f\vert^{q}d\mu+D'\int\vert f\vert^{q}d\mu.}
\end{array}
\]
\end{proof}

\section{Poincaré Inequality}

We now have the U-Bound (\ref{eq:u}) at our disposal and are ready
to prove the q-Poincaré inequality using the method \cite{key-6}:

Let $\lambda$ be a measure satisfying the q-Poincaré inequality for
every ball ${\displaystyle B_{R}=\{x:N\left(x\right)<R\},~}$ i.e.
there exists a constant $C_{R}\in\left(0,\infty\right)$ such that
\[
\frac{1}{\vert B_{R}\vert}\int_{B_{R}}\left|f-\frac{1}{\vert B_{R}\vert}\int_{B_{R}}f\right|^{q}d\lambda\leq C_{R}\frac{1}{\vert B_{R}\vert}\int_{B_{R}}\vert\triangledown f\vert^{q}d\lambda,
\]

where $1\leq q<\infty.$

Note that we have this Poincaré inequality on balls in the setting
of Nilpotent lie groups thanks to J. Jerison's celebrated paper \cite{key-18}.
With this we can use the following result: 
\begin{thm}[Hebisch, Zegarli\'{n}ski \cite{key-6}]
 Let $\mu$ be a probability measure on $\mathbb{R}^{n}$ which is
absolutely continuous with respect to the measure $\lambda$ and such
that
\[
\int f^{q}\eta d\mu\leq C\int\vert\triangledown f\vert^{q}d\mu+D\int f^{q}d\mu
\]
with some non-negative function $\eta$ and some constants $C,D\in\left(0,\infty\right)$
independent of a function $f.$ If for any $L\in\left(0,\infty\right)$
there is a constant $A_{L}$ such that ${\displaystyle \frac{1}{A_{L}}\leq\frac{d\mu}{d\lambda}\leq A_{L}}$
on the set $\left\{ \eta<L\right\} $ and, for some $R\in\left(0,\infty\right)$
(depending on L), we have $\left\{ \eta<L\right\} \subset B_{R},$
then $\mu$ satisfies the q-Poincaré inequality
\[
\mu\vert f-\mu f\vert^{q}\leq c\mu\vert\triangledown f\vert^{q}
\]
with some constant $c\in\left(0,\infty\right)$ independent of $f.$
\end{thm}

The role of $\eta$ in Theorem 3 is played by ${\displaystyle \frac{g'(N)}{N^{2}}}$
from the U-Bound of Theorem 1. Hence, we get the following corollaries: 
\begin{cor}
The Poincaré inequality for $q\geq2$ holds for the measure ${\displaystyle d\mu=\frac{exp\left(-cosh\left(N^{k}\right)\right)}{Z}d\lambda},$
where $\lambda$ is the Lebesgue measure, and $k\geq1$ in the setting
of the step-two Carnot group. 
\end{cor}

\begin{proof}
$g\left(N\right)=cosh\left(N^{k}\right),$ so $g'\left(N\right)=kN^{k-1}sinh\left(N^{k}\right),$
and\\
 ${\displaystyle g''(N)=k(k-1)N^{k-2}sinh(N^{k})+k^{2}N^{2k-2}cosh(N^{k})}.$

First, on $\{N\geq1\},$ ${\displaystyle g''(N)\leq k^{3}N^{3k}sinh^{3}(N^{k})=g'(N)^{3}N^{3},}$
so the condition of Theorem 1 is satisfied. Second,
\[
\begin{array}{cl}
{\displaystyle {\displaystyle \int\frac{g'\left(N\right)}{N^{2}}f^{q}d\mu}} & {\displaystyle {\displaystyle {\displaystyle =\int f^{q}\left[kN^{k-3}sinh\left(N^{k}\right)\right]d\mu}}}\\
\\
 & {\displaystyle =\int_{\left\{ N<1\right\} }f^{q}\left[kN^{k-3}\frac{e^{N^{k}}-e^{-N^{k}}}{2}\right]d\mu+\int_{\left\{ N\geq1\right\} }f^{q}\left[kN^{k-3}\frac{e^{N^{k}}-e^{-N^{k}}}{2}\right]d\mu}\\
\\
 & {\displaystyle \leq\int_{\left\{ N<1\right\} }f^{q}\left[k\frac{e-e^{-1}}{2}\right]d\mu+C^{'}\int_{\left\{ N\geq1\right\} }\vert\triangledown f\vert^{q}d\mu+D^{'}\int_{\left\{ N\geq1\right\} }\vert f\vert^{q}d\mu}\\
\\
 & {\displaystyle \leq C\int\vert\triangledown f\vert^{q}d\mu+D\int\vert f\vert^{q}d\mu.}
\end{array}
\]

Thus, the conditions of Theorem 3 are satisfied for $\eta=kN^{k-3}sinh\left(N^{k}\right)$
, and $k\geq1.$ So, the Poincaré inequality holds for $q\geq2.$
\end{proof}
The following corollary was proven in the setting of the Heisenberg-type
group for the Kaplan norm $N=\left(|x|^{4}+16|z|^{2}\right)^{\frac{1}{4}}$
by J. Inglis, Theorem 4.5.5 of \cite{key-7}. 

In the setting of the step-two Carnot group, we obtain a generalised
version for a similar homogeneous norm $N=\left(|x|^{4}+a|z|^{2}\right)^{\frac{1}{4}}$. 
\begin{cor}
The Poincaré inequality for $q\geq2$ holds for the measure ${\displaystyle d\mu=\frac{exp\left(-N^{k}\right)}{Z}d\lambda,}$
where $\lambda$ is the Lebesgue measure, and $k\geq4$ in the setting
of the step-two Carnot group.
\end{cor}

\begin{proof}
Let $g(N)=N^{k},$ so $g^{'}\left(N\right)=kN^{k-1},$ and $g''(N)=k(k-1)N^{k-2}.$
First, on $\{N\geq1\},$ ${\displaystyle g''(N)\leq k^{3}N^{3k}=g'(N)^{3}N^{3},}$
so the condition of Theorem 1 is satisfied. Secondly, 
\[
\begin{array}{cl}
{\displaystyle \int\frac{g'\left(N\right)}{N^{2}}f^{q}d\mu} & {\displaystyle =\int f^{q}\left[kN^{k-3}\right]d\mu}\\
\\
 & {\displaystyle =\int_{\left\{ N<1\right\} }f^{q}\left[kN^{k-3}\right]d\mu+\int_{\left\{ N\geq1\right\} }f^{q}\left[kN^{k-3}\right]d\mu}\\
\\
 & {\displaystyle \leq\int_{\left\{ N<1\right\} }kf^{q}d\mu+C^{'}\int_{\left\{ N\geq1\right\} }\vert\triangledown f\vert^{q}d\mu+D^{'}\int_{\left\{ N\geq1\right\} }\vert f\vert^{q}d\mu}\\
\\
 & {\displaystyle \leq C\int\vert\triangledown f\vert^{q}d\mu+D\int\vert f\vert^{q}d\mu.}
\end{array}
\]
Thus, the conditions of Theorem 3 are satisfied for $\eta=kN^{k-3},$
and $k\geq4.$ So, the Poincaré inequality holds for $q\geq2.$
\end{proof}
The following corollary improves Corollary 5 in an interesting way.
Namely, at a cost of a logarithmic factor, we now get the Poincaré
inequality for polynomial growth of order $k\geq3.$ 
\begin{cor}
The Poincaré inequality for $q\geq2$ holds for the measure ${\displaystyle d\mu=\frac{exp\left(-N^{k}log\left(N+1\right)\right)}{Z}d\lambda},$
where $\lambda$ is the Lebesgue measure, and $k\geq3$ in the setting
of the step-two Carnot group. 
\end{cor}

\begin{proof}
Let $g(N)=N^{k}log\left(N+1\right),$ so ${\displaystyle g^{'}\left(N\right)=kN^{k-1}log(N+1)+\frac{N^{k}}{N+1},}$
and
\[
{\displaystyle g''(N)=k(k-1)N^{k-2}log(N+1)+\frac{2kN^{k-1}}{N+1}-\frac{N^{k}}{(N+1)^{2}}.}
\]
First, on $\{N\geq1\},$ 
\[
g''(N)\leq k^{3}N^{3k}log^{3}(N+1)+\frac{N^{3k+3}}{(N+1)^{3}}+\frac{3k^{2}N^{3k+1}log^{2}(N+1)}{N+1}+\frac{3kN^{3k+2}log(N+1)}{(N+1)^{2}}=N^{3}g'(N)^{3},
\]
so the condition of Theorem 1 is satisfied. Secondly,
\[
\begin{array}{cl}
{\displaystyle \int\frac{g'\left(N\right)}{N^{2}}f^{q}d\mu} & {\displaystyle =\int f^{q}\left[kN^{k-3}log(N+1)+\frac{N^{k-2}}{N+1}\right]d\mu}\\
\\
 & {\displaystyle =\int_{\left\{ N<1\right\} }f^{q}\left[kN^{k-3}log(N+1)+\frac{N^{k-2}}{N+1}\right]d\mu+\int_{\left\{ N\geq1\right\} }f^{q}\left[kN^{k-3}log(N+1)+\frac{N^{k-2}}{N+1}\right]d\mu}\\
\\
 & {\displaystyle \leq\int_{\left\{ N<1\right\} }\left(klog(2)+1\right)f^{q}d\mu+C^{'}\int_{\left\{ N\geq1\right\} }\vert\triangledown f\vert^{q}d\mu+D^{'}\int_{\left\{ N\geq1\right\} }\vert f\vert^{q}d\mu}\\
\\
 & {\displaystyle \leq C\int\vert\triangledown f\vert^{q}d\mu+D\int\vert f\vert^{q}d\mu}.
\end{array}
\]
Thus, the conditions of Theorem 3 are satisfied for ${\displaystyle \eta=kN^{k-3}log(N+1)+\frac{N^{k-2}}{N+1},}$
and $k\geq3.$ Hence, the Poincaré inequality holds for $q\geq2.$ 
\end{proof}

\section{$\phi-$Logarithmic Sobolev Inequality }

After proving the q-Poincaré inequality for measures as a function
of the homogeneous norm $N=\left(|x|^{4}+a|z|^{2}\right)^{\frac{1}{4}}$,
a natural question would be if one could obtain other coercive inequalities.
J.Inglis et al.'s Theorem 2.1 \cite{key-8} proved that, for the measure
${\displaystyle d\mu=\frac{e^{-U}d\lambda}{Z},}$ provided we have
the U-Bound
\[
{\displaystyle \mu(|f|(|U|^{\beta}+|\triangledown U|))\leq A\mu|\triangledown f|+B\mu|f|,}
\]
 one obtains
\[
\mu\left(|f|\left|log\frac{|f|}{\mu|f|}\right|^{\beta}\right)\leq C\mu|\triangledown f|+B\mu|f|.
\]
We will first extend their theorem, and then we will use Theorem 1
to get more general coercive inequalities. 
\begin{thm}
Let U be a locally lipschitz function on $\mathbb{R}^{N}$ which is
bounded below such that $Z=\int e^{-U}d\lambda<\infty$ and ${\displaystyle d\mu=\frac{e^{-U}}{Z}d\lambda.}$
Let $\phi:[0,\infty)\rightarrow\mathbb{R}^{+}$ be a non-negative,
non-decreasing, concave function such that $\phi(0)>0,$ and $\phi'(0)>0.$
Assume the following classical Sobolev inequality is satisfied:
\[
\left(\int|f|^{q+\epsilon}d\lambda\right)^{\frac{q}{q+\epsilon}}\leq a\int|\triangledown f|^{q}d\lambda+b\int|f|^{q}d\lambda
\]
 for some $a,$ $b\in[0,\infty),$ and $\epsilon>0.$ Moreover, if
for some $A,$ $B\in[0,\infty),$ we have:
\[
\mu\left(|f|^{q}(\phi(U)+|\triangledown U|^{q})\right)\leq A\mu|\triangledown f|^{q}+B\mu|f|^{q},
\]
Then, there exists constants $C,$ $D\in[0,\infty)$ such that:
\[
\mu\left(|f|^{q}\phi\left(\left|log\frac{|f|^{q}}{\mu|f|^{q}}\right|\right)\right)\leq C\mu|\triangledown f|^{q}+D\mu|f|^{q},
\]
 for all locally Lipschitz functions $f.$ 
\end{thm}

\begin{proof}
First of all, we remark that for a concave function $\phi$ as in
our assumptions, we have
\begin{equation}
\phi(y)-\phi(x)\leq\phi'(0)(y-x).\label{eq:15}
\end{equation}
 Suppose first that $\int|f|^{q}=1,$ and let $E=\{x\in\mathbb{R^{N}}:\;|log|f|^{q}|>U\}.$ 

\[
\begin{array}{cl}
{\displaystyle \int|f|^{q}\phi\left(\left|log|f|^{q}\right|\right)d\mu} & {\displaystyle =\int_{E}|f|^{q}\phi\left(\left|log|f|^{q}\right|\right)d\mu+\int_{E^{c}}|f|^{q}\phi\left(\left|log|f|^{q}\right|\right)d\mu}\\
\\
 & {\displaystyle =\int_{E}|f|^{q}\left[\phi\left(\left|log|f|^{q}\right|\right)-\phi(U)\right]d\mu+\int_{E}|f|^{q}\phi(U)d\mu+\int_{E^{c}}|f|^{q}\phi\left(\left|log|f|^{q}\right|\right)d\mu}\\
\\
 & {\displaystyle \leq\phi'(0)\int_{E}|f|^{q}\left(\left|log|f|^{q}\right|-U\right)d\mu+\int_{E}|f|^{q}\phi(U)d\mu+\int_{E^{c}}|f|^{q}\phi(U)d\mu,}
\end{array}
\]
where the last inequality uses (\ref{eq:15}) on $E,$ and uses the
fact that $\phi$ is non-decreasing on $E^{c},$ hence, $|log|f|^{q}|<U.$
Let $E_{1}=\{log|f|^{q}>U\},$ $E_{2}=\{log|f|^{q}<-U\},$ and $c=\int_{E_{1}}|f|^{q}e^{-U}d\lambda.$
\[
\int|f|^{q}\phi\left(\left|log|f|^{q}\right|\right)d\mu\leq\frac{c\phi'(0)q}{\epsilon Z}\int_{E_{1}}\frac{\left(|f|e^{-\frac{U}{q}}\right)^{q}}{c}log\left(|f|e^{-\frac{U}{q}}\right)^{\epsilon}d\lambda+\phi'(0)\int_{E_{2}}e^{-U}d\mu+\int|f|^{q}\phi(U)d\mu
\]
Using Jensen's inequality,
\[
\begin{array}{cl}
 & {\displaystyle \leq\frac{c\phi'(0)(q+\epsilon)}{Z\epsilon}log\left(\int_{E_{1}}\frac{\left(|f|e^{-\frac{U}{q}}\right)^{q+\epsilon}}{c}d\lambda\right)^{\frac{q}{q+\epsilon}}+\phi'(0)\int_{E_{2}}1d\mu+\int|f|^{q}\phi(U)d\mu}\\
 & {\displaystyle \leq\frac{c^{\frac{\epsilon}{q+\epsilon}}\phi'(0)(q+\epsilon)}{Z\epsilon}log\left(\int\left(|f|e^{-\frac{U}{q}}\right)^{q+\epsilon}d\lambda\right)^{\frac{q}{q+\epsilon}}+\phi'(0)Z+\int|f|^{q}\phi(U)d\mu}
\end{array}
\]
 Using classical Sobolev inequality, 
\[
\begin{array}{cl}
 & {\displaystyle \leq a\int|f|^{q}d\mu+b\int\left|\triangledown(fe^{-\frac{U}{q}})\right|^{q}d\lambda+\phi'(0)Z+\int|f|^{q}\phi(U)d\mu}\\
\\
 & {\displaystyle =a+\phi'(0)Z+b\int\left|(\triangledown f)e^{-\frac{U}{q}}-\frac{f}{q}(\triangledown U)e^{-\frac{U}{q}}\right|^{q}d\lambda+\int|f|^{q}\phi(U)d\mu}\\
\\
 & {\displaystyle \leq a+\phi'(0)Z+b2^{q-1}Z\int|\triangledown f|^{q}d\mu+\frac{Zb2^{q-1}}{q^{q}}\int|f|^{q}|\triangledown U|^{q}d\mu+\int|f|^{q}\phi(U)d\mu}\\
\\
 & {\displaystyle \leq a+\phi'(0)Z+b2^{q-1}Z\int|\triangledown f|^{q}d\mu+max\{\frac{Zb2^{q-1}}{q^{q}},1\}\int|f|^{q}|\left(\triangledown U|^{q}+\phi(U)\right)d\mu}
\end{array}
\]
Using the U-bound in the Theorem's condition
\[
\leq A+B\int|\triangledown f|^{q}d\mu
\]
Finally, replace $|f|^{q}$ by ${\displaystyle \frac{|f|^{q}}{\mu|f|^{q}}}$
to get the desired inequality. 
\end{proof}
\begin{cor}
$\phi(x)=(1+x)^{\beta},$ for $\beta\in(0,1]$ is non-negative, non-decreasing,
and concave function satisfying $\phi(0)=1>0,$ and $\phi'(0)=\beta>0.$
Therefore, Theorem 7 applies, and

\[
\mu\left(|f|^{q}\left|log\left(\frac{|f|^{q}}{\mu|f|^{q}}\right)\right|^{\beta}\right)\leq\mu\left(|f|^{q}\left(1+\left|log\left(\frac{|f|^{q}}{\mu|f|^{q}}\right)\right|\right)^{\beta}\right)\leq C\mu|\triangledown f|^{q}+D\mu|f|^{q}.
\]
\end{cor}

\begin{cor}
Let $h^{(1)}(x)=log(\alpha+x),$ where $\alpha>1.$ Define recursively

$h^{(n)}(x)=log(\alpha+h^{(n-1)}(x)).$ Then, for all $n\geq1,$ $h^{(n)}(x)=\phi(x)$
of Theorem 7. Therefore, we obtain
\[
\mu\left(|f|^{q}log^{*(n)}\left(\frac{|f|^{q}}{\mu|f|^{q}}\right)\right)\leq\mu\left(|f|^{q}h^{(n)}\left(\frac{|f|^{q}}{\mu|f|^{q}}\right)\right)\leq C\mu|\triangledown f|^{q}+D\mu|f|^{q},
\]
where $log^{*(n)}$ is the positive part of $log^{(n)}.$ 
\end{cor}

\begin{proof}
The proof proceeds by induction. For $n=1,$ ${\displaystyle h^{(1)}(x)=log(\alpha+x),}$
so ${\displaystyle h^{(1)}(x)'=\frac{1}{\alpha+x},}$ and ${\displaystyle h^{(1)}(x)''=\frac{-1}{(\alpha+x)^{2}}.}$
${\displaystyle h^{(1)}(0)=log(\alpha)>0,}$ and ${\displaystyle h^{(1)}(0)'=\frac{1}{\alpha}>0.}$
In addition, $h^{(1)}(x)$ is non-negative, non-decreasing, and concave;
hence the conditions of Theorem 7 are satisfied. 

Assume it is true for $n=k,$ prove it is true for $n=k+1:$ ${\displaystyle h^{(k+1)}(x)=log(\alpha+h^{(k)}(x)),}$
so ${\displaystyle h^{(k+1)}(x)'=\frac{h^{(k)}(x)'}{\alpha+h^{(k)}(x)},}$
and ${\displaystyle h^{(k+1)}(x)''=\frac{h^{(k)}(x)''}{\alpha+h^{(k)}(x)}-\frac{h^{(k)}(x)'}{(\alpha+h^{(k)}(x))^{2}}.}$
The result follows directly. 
\end{proof}
Returning to the measure as a function of the homogeneous norm $N=\left(|x|^{4}+a|z|^{2}\right)^{\frac{1}{4}}$,
${\displaystyle d\mu=\frac{e^{-N^{p}}}{Z}d\lambda,}$ we will prove
using Theorems 1 and 7, that the Log$^{\beta}$-Sobolev inequality
$(0<\beta\leq1)$ (Corollary 8) holds for $q\geq2,$ yet fails for
$1<q<{\displaystyle \frac{2p\beta}{p-1}}.$ To start with, we will
show why the Log$^{\beta}$-Sobolev inequality fails for $1<q<{\displaystyle \frac{2p\beta}{p-1}}.$
The proof uses the idea of Theorem 6.3 of \cite{key-6}. 
\begin{thm}
Let $\mathbb{G}$ be a stratified group, and $N$ be a smooth homogenous
norm on $\mathbb{G}.$ For $\alpha>0,\;p\geq1,$ let ${\displaystyle d\mu=\frac{e^{-\alpha N^{p}}}{Z}d\lambda,}$
where $Z$ is the normalization constant. The measure $\mu$ satisfies
no Log$^{\beta}$-Sobolev inequality $(0<\beta\leq1)$ for $1<q<{\displaystyle \frac{2p\beta}{p-1}}.$
\end{thm}

\begin{proof}
The proof is by contradiction. Let $x_{0}$ be such that $(\triangledown N)(x_{0})=0.$
For $t>0$ put ${\displaystyle r=t^{\frac{-p+1}{2}},}$ and

${\displaystyle f=max\left[min\left(\frac{2-d(x,tx_{0})}{r},1\right),0\right].}$
On $B(tx_{0},2r)=\{x:d(x,tx_{0})\leq2r\},$ by homogeneity, by Lemma
6.3 of \cite{key-6}, and by the fact that $(\triangledown N)(x_{0})=0,$
we have $|N(x)-N(tx_{0})|\leq c_{1}r^{2},$ so $|N(x)^{p}-N(tx_{0})^{p}|\leq c_{2}.$
Consequently, the exponential factor in $\mu$ is comparable to a
constant on the support of $f.$ Also, $|\triangledown f|\leq\frac{1}{r},$
and
\begin{equation}
\mu|f|^{q}\approx r^{Q}\left(e^{-\alpha N^{p}(tx_{0})}\right)\label{eq:16}
\end{equation}
\begin{equation}
log(\mu|f|^{q})\approx-t^{p}\label{eq:17}
\end{equation}
\begin{equation}
\mu|\triangledown f|^{q}\approx r^{-q}r^{Q}e^{-\alpha N^{p}(tx_{0})}\label{eq:18}
\end{equation}
Choose $t$ large enough so that $r<\frac{2}{3}.$ On $B(tx_{0},2r),$
$2-d(x,tx_{0})\geq2-2r\geq r.$ Thus, we have ${\displaystyle f=max\left[min\left(\frac{2-d(x,tx_{0})}{r},1\right),0\right]=1,}$
and consequently $log|f|^{q}=0.$ 

\[
\mu\left(|f|^{q}\left|log\left(\frac{|f|^{q}}{\mu|f|^{q}}\right)\right|^{\beta}\right)=\mu\left(|f|^{q}\left(\left|log|f|^{q}-log\mu|f|^{q}\right|^{\beta}\right)\right)=\mu\left(|f|^{q}\left(\left|log\mu|f|^{q}\right|^{\beta}\right)\right)
\]
 by (\ref{eq:17}),
\[
\approx\mu(|f|^{q}t^{p\beta})
\]
by (\ref{eq:16})
\[
\approx t^{p\beta}r^{Q}e^{-\alpha N^{p}(tx_{0})}.
\]
Assuming we have $\beta-$logarithmic Sobolev inequality, and using
(\ref{eq:18}), we get: 
\[
t^{p\beta}r^{Q}e^{-\alpha N^{p}(tx_{0})}\leq Mr^{-q}r^{Q}e^{-\alpha N^{p}(tx_{0})}
\]
since ${\displaystyle r=t^{\frac{-p+1}{2}},}$ 
\[
t^{p\beta}\leq Mt^{-q(\frac{-p+1}{2})}.
\]
For $t$ large enough, we get a contradiction when $p\beta>{\displaystyle \frac{q(p-1)}{2}}$
i.e. for $q<{\displaystyle \frac{2p\beta}{p-1}}$. So, the measure
$\mu$ satisfies no $\beta-$logarithmic Sobolev inequality for $1<q<{\displaystyle \frac{2p\beta}{p-1}}.$ 
\end{proof}
Now we prove that for $q\geq2$, Log$^{\beta}$-Sobolev inequality
holds true for ${\displaystyle d\mu=\frac{e^{-\alpha N^{p}}}{Z}d\lambda},$
where $N=\left(|x|^{4}+a|z|^{2}\right)^{\frac{1}{4}}$ and ${\displaystyle 0<\beta\leq\frac{p-3}{p}.}$
\begin{thm}
Let $\mathbb{G}$ be an step-two Carnot group. Consider the probability
measure given by
\[
{\displaystyle d\mu=\frac{e^{-g(N)}}{Z}d\lambda,}
\]
where $Z$ is the normalization constant and $N=\left(|x|^{4}+a|z|^{2}\right)^{\frac{1}{4}}$
with $a\in(0,\infty)$. Let $g:\left[0,\infty\right)\rightarrow\left[0,\infty\right)$
be a differentiable increasing function such that $g'(N)$ is increasing,
${\displaystyle g(N)\leq\left(c\frac{g'(N)}{N^{2}}\right)^{\frac{1}{\beta}}},$
and $g''(N)<dg'(N)^{2}$ on $\{N\geq1\},$ for some constants $c,d\in(0,\infty).$
Then 
\[
\mu\left(|f|^{q}\left|log\left(\frac{|f|^{q}}{\mu|f|^{q}}\right)\right|^{\beta}\right)\leq C\mu|f|^{q}+D\mu|\triangledown f|^{q},
\]

for $C$ and $D$ positive constants and for $q\geq2.$
\end{thm}

\begin{proof}
On $\{N\geq1\},$ $g''(N)<dg'(N)^{2}\leq g'(N)^{3}N^{3},$ so the
condition of Theorem 1 is satisfied. Thus, on $\{N\geq1\},$ we have
the U-bound (\ref{eq:u}):
\[
\mu\left(\frac{g'(N)}{N^{2}}|f|^{q}\right)\leq C\mu|\triangledown f|^{q}+D\mu|f|^{q}.
\]
By the condition ${\displaystyle g(N)\leq\left(c\frac{g'(N)}{N^{2}}\right)^{\frac{1}{\beta}},}$
we obtain ${\displaystyle \phi(g(N))=(1+g(N))^{\beta}\leq\frac{g'(N)}{N^{2}}}$
on $\{N\geq1\}$. Hence, since $g(N)$ is increasing and using the
U-bound, we have
\begin{equation}
\begin{array}{cl}
{\displaystyle \mu\left(\phi(g(N))|f|^{q}\right)} & {\displaystyle \leq\int_{\{N\geq1\}}\left(\frac{g'(N)}{N^{2}}|f|^{q}\right)d\mu+\int_{\{N<1\}}\phi(g(N))|f|^{q}d\mu}\\
\\
 & \leq{\displaystyle \int_{\{N\geq1\}}\left(\frac{g'(N)}{N^{2}}|f|^{q}\right)d\mu+\int_{\{N<1\}}(1+g(1))^{\beta}|f|^{q}d\mu}\\
\\
 & {\displaystyle \leq C\mu|\triangledown f|^{q}+D\mu|f|^{q}.}
\end{array}\label{eq:19-1}
\end{equation}

In order to use Theorem 7, it remains to prove:
\begin{equation}
\mu\left(|f|^{q}|\triangledown g(N)|^{q}\right)\leq C\mu|f|^{q}+D\mu|\triangledown f|^{q}.\label{eq:20-1}
\end{equation}

On $\{N<1\},$ since $g'(N)$ is increasing and using (\ref{eq:1-1}),
\[
\int_{\{N<1\}}\left(|f|^{q}|\triangledown g(N)|^{q}\right)d\mu=\int_{\{N<1\}}\left(|f|^{q}|g'(N)\triangledown N|^{q}\right)d\mu\leq C^{\frac{q}{2}}\int_{\{N<1\}}|f|^{q}|g'(1)|^{q}d\mu.
\]
We now need to consider $\{N\geq1\}:$
\begin{equation}
\int|f|^{q}\left(\triangledown g(N)\cdot V-\triangledown\cdot V\right)d\mu=\int\triangledown|f|^{q}\cdot Vd\mu\leq\frac{\epsilon}{p}\int|f|^{q}|V|^{p}d\mu+\frac{1}{\epsilon^{\frac{q}{p}}}q^{q-1}\int|\triangledown f|^{q}d\mu,\label{eq:21-1}
\end{equation}
where the last inequality uses $\text{\ensuremath{{\displaystyle ab \leq\epsilon\frac{a^{p}}{p} + \frac{b^{q}}{\epsilon^{\frac{q}{p}}q} ,}}}$
where $a=|f|^{q-1}|V|,$ and $b=q|\triangledown f|.$ Let ${\displaystyle V=\triangledown N\frac{|x|^{q-2}}{N^{q-2}}g'(N)^{q-1}.}$
Since $\triangledown g(N)=g'(N)\triangledown N,$ then $\triangledown g(N)\cdot V=|\triangledown g(N)|^{q},$
which is the term on the left hand side of (\ref{eq:20-1}). Using
the inequality (\ref{eq:1-1}) on the first term on the right hand
side of (\ref{eq:21-1}) we get
\[
\begin{array}{cl}
{\displaystyle \frac{\epsilon}{p}\int|f|^{q}|V|^{p}d\mu} & {\displaystyle =\frac{\epsilon}{p}\int|\triangledown N|^{p}|f|^{q}\frac{|x|^{(q-2)p}g'(N)^{q}}{N^{(q-2)p}}d\mu}\\
\\
 & \leq{\displaystyle \frac{\epsilon C^{\frac{p}{2}}}{p}\int|f|^{q}\frac{|x|^{q}g'(N)^{q}}{N^{q}}d\mu}
\end{array}
\]

which can subtracted from the left hand side of (\ref{eq:21-1}) since
by choosing $\epsilon$ small enough and noting that using (\ref{eq:1-1}),
one has
\[
\triangledown g(N)\cdot V=|\triangledown N|^{2}\frac{|x|^{q-2}}{N^{q-2}}g'(N)^{q}\geq A\frac{|x|^{q}}{N^{q}}g'(N)^{q}.
\]

It remains to compute $\triangledown\cdot V.$ Using ${\displaystyle |\Delta N|\leq B\frac{|x|^{2}}{N^{3}},}$
(\ref{eq:2-1}), and ${\displaystyle \frac{x}{|x|}\cdot\nabla N=\frac{|x|^{3}}{N^{3}},}$
(\ref{eq:3}), we have
\[
\triangledown\cdot V=\Delta N\frac{|x|^{q-2}}{N^{q-2}}g'(N)^{q-1}+(q-2)\frac{|x|^{q}g'(N)^{q-1}}{N^{q+1}}-(q-2)\frac{|\triangledown N|^{2}|x|^{q-2}g'(N)^{q-1}}{N^{q-1}}
\]
\[
+(q-1)g'(N)^{q-2}g''(N)\frac{|x|^{q-2}|\triangledown N|^{2}}{N^{q-2}}
\]
and hence

\[
|\triangledown\cdot V|\leq B\frac{|x|^{q}}{N^{q+1}}g'(N)^{q-1}+(q-2)\frac{|x|^{q}g'(N)^{q-1}}{N^{q+1}}+(q-2)\frac{C|x|^{q}g'(N)^{q-1}}{N^{q+1}}+C(q-1)g'(N)^{q-2}g''(N)\frac{|x|^{q}}{N^{q}}.
\]
 All terms can be absorbed by the first term in (\ref{eq:21-1}).
Using (\ref{eq:19-1}) and (\ref{eq:20-1}), the condition of Theorem
7 is satisfied, and we obtain Log$^{\beta}$-Sobolev inequality:
\[
\mu\left(|f|^{q}\left|log\left(\frac{|f|^{q}}{\mu|f|^{q}}\right)\right|^{\beta}\right)\leq C\mu|f|^{q}+D\mu|\triangledown f|^{q}
\]
 for $C$ and $D$ positive constants. 
\end{proof}
\begin{cor}
Let $\mathbb{G}$ be a step-two Carnot group and $N=\left(|x|^{4}+a|z|^{2}\right)^{\frac{1}{4}}$
with $a\in(0,\infty)$ . Let the probability measure be ${\displaystyle d\mu=\frac{e^{-\beta N^{p}}}{Z}d\lambda,}$
where $Z$ is the normalization constant. Then, for $p\geq4$ and
${\displaystyle 0<\beta\leq\frac{p-3}{p},}$\[
\mu\left(|f|^{q}\left|log\left(\frac{|f|^{q}}{\mu|f|^{q}}\right)\right|^{\beta}\right)\leq C\mu|f|^{q}+D\mu|\triangledown f|^{q},
\]
 for $C$ and $D$ positive constants and for $q\geq2.$ 
\end{cor}

\section{Appendix: Proof of Lemma 2}
\begin{proof}
We first compute $\triangledown N=(X_{i}N)_{i=1,...,n}.$

\[
X_{i}N=N^{-3}\left(|x|^{2}x_{i}+\frac{a}{4}\sum_{k=1}^{m}\sum_{l=1}^{n}\Lambda_{il}^{\left(k\right)}x_{l}z_{k}\right).
\]
Therefore,
\begin{equation}
\begin{array}{ll}
\left|\nabla N\right|^{2} & ={\displaystyle N^{-6}\left(|x|^{6}+\frac{a}{2}\sum_{i=1}^{n}\sum_{k=1}^{m}\sum_{l=1}^{n}\Lambda_{il}^{\left(k\right)}|x|^{2}x_{i}x_{l}z_{k}+\sum_{i=1}^{n}\frac{a^{2}}{16}\sum_{k,k'=1}^{m}\sum_{l,l'=1}^{n}\Lambda_{il}^{\left(k\right)}\Lambda_{il'}^{\left(k'\right)}x_{l}x_{l'}z_{k}z_{k'}\right)}\\
 & {\displaystyle =N^{-6}\left(|x|^{6}+\frac{a^{2}}{16}\sum_{i=1}^{n}\sum_{k,k'=1}^{m}\sum_{l,l'=1}^{n}\Lambda_{il}^{\left(k\right)}\Lambda_{il'}^{\left(k'\right)}x_{l}x_{l'}z_{k}z_{k'}\right)}\\
 & ={\displaystyle \frac{|x|^{2}}{N^{2}}N^{-4}\left(|x|^{4}+\frac{a^{2}}{16}\sum_{i=1}^{n}\sum_{k,k'=1}^{m}\sum_{l,l'=1}^{n}\Lambda_{il}^{\left(k\right)}\Lambda_{il'}^{\left(k'\right)}\frac{x_{l}}{|x|}\frac{x_{l'}}{|x|}z_{k}z_{k'}\right),}
\end{array}\label{eq:4-1}
\end{equation}
where we used that for each skew-symmetric matrix $\Lambda^{(k)},$
all $k\in\{1,...,m\},$ we have that
\[
\sum_{l=1}^{n}\sum_{i=1}^{n}\Lambda_{il}^{\left(k\right)}x_{l}x_{i}=0.
\]
From (\ref{eq:4-1}), that with some constants $A,C\in(0,\infty),$
we have
\[
A\frac{|x|^{2}}{N^{2}}\leq\left|\nabla N\right|^{2}\leq C\frac{|x|^{2}}{N^{2}}.
\]
By choosing $a\in(0,\infty)$ sufficiently small, we can ensure that
$C\leq1.$ We note that using antisymmetry of matrices $\Lambda_{il}^{\left(k\right)}$
we get
\[
\frac{x}{|x|}\cdot\nabla N=\sum_{i=1}^{n}\frac{x_{i}}{|x|}N^{-3}\left(|x|^{2}x_{i}+\frac{a}{4}\sum_{k=1}^{m}\sum_{l=1}^{n}\Lambda_{il}^{\left(k\right)}x_{l}z_{k}\right)
\]

Next we compute
\[
\begin{array}{cl}
X_{i}^{2}N & ={\displaystyle \left(\frac{\partial}{\partial x_{i}}+\frac{1}{2}\sum_{k=1}^{m}\sum_{l=1}^{n}\Lambda_{il}^{\left(k\right)}x_{l}\frac{\partial}{\partial z_{k}}\right)\left(N^{-3}\left(|x|^{2}x_{i}+\frac{a}{4}\sum_{k=1}^{m}\sum_{l=1}^{n}\Lambda_{il}^{\left(k\right)}x_{l}z_{k}\right)\right)}\\
 & {\displaystyle =-3\left(N^{-7}\left(|x|^{2}x_{i}+\frac{a}{4}\sum_{k=1}^{m}\sum_{l=1}^{n}\Lambda_{il}^{\left(k\right)}x_{l}z_{k}\right)^{2}\right)}\\
 & {\displaystyle +\left(N^{-3}\left(|x|^{2}+2x_{i}^{2}+\frac{a}{4}\sum_{k=1}^{m}\sum_{l=1}^{n}\Lambda_{il}^{\left(k\right)}\delta_{il}z_{k}\right)\right)}\\
 & {\displaystyle +\left(N^{-3}\left(\frac{1}{2}\sum_{k=1}^{m}\sum_{l=1}^{n}\Lambda_{il}^{\left(k\right)}x_{l}\frac{a}{4}\sum_{k'=1}^{m}\sum_{l'=1}^{n}\Lambda_{il'}^{\left(k'\right)}x_{l'}\delta_{kk'}\right)\right).}
\end{array}
\]

Hence we obtain
\[
\begin{array}{cl}
\Delta N & {\displaystyle =\sum_{i=1}^{n}X_{i}^{2}N}\\
 & {\displaystyle =-3\left(N^{-7}\left(|x|^{6}+2\frac{a}{4}\sum_{k=1}^{m}\sum_{i=1}^{n}\sum_{l=1}^{n}\Lambda_{il}^{\left(k\right)}|x|^{2}x_{i}x_{l}z_{k}\right)\right.}\\
 & {\displaystyle +\left.\frac{a^{2}}{16}\sum_{k=1}^{m}\sum_{l=1}^{n}\sum_{k'=1}^{m}\sum_{l'=1}^{n}\sum_{i=1}^{n}\Lambda_{il}^{\left(k\right)}\Lambda_{il'}^{\left(k'\right)}x_{l'}x_{l}z_{k}z_{k'}\right)}\\
 & {\displaystyle +\left(N^{-3}\left((n+2)|x|^{2}+\frac{a}{4}\sum_{k=1}^{m}\sum_{i=1}^{n}\sum_{l=1}^{n}\Lambda_{il}^{\left(k\right)}\delta_{il}z_{k}\right)\right)}\\
 & {\displaystyle +\left(N^{-3}\left(\frac{1}{2}\sum_{k=1}^{m}\sum_{l=1}^{n}\sum_{i=1}^{n}\Lambda_{il}^{\left(k\right)}x_{l}\frac{a}{4}\sum_{k'=1}^{m}\sum_{l'=1}^{n}\Lambda_{il'}^{\left(k'\right)}x_{l'}\delta_{kk'}\right)\right).}
\end{array}
\]

which after simplifications yields
\[
\begin{array}{cl}
\Delta N & {\displaystyle =\sum_{i=1}^{n}X_{i}^{2}N}\\
 & {\displaystyle =-3\left(N^{-7}\left(|x|^{6}+\frac{a^{2}}{16}\sum_{k=1}^{m}\sum_{l=1}^{n}\sum_{k'=1}^{m}\sum_{l'=1}^{n}\sum_{i=1}^{n}\Lambda_{il}^{\left(k\right)}\Lambda_{il'}^{\left(k'\right)}x_{l'}x_{l}z_{k}z_{k'}\right)\right)}\\
 & {\displaystyle +\left(N^{-3}\left((n+2)|x|^{2}\right)\right)}\\
 & {\displaystyle +\left(N^{-3}\left(\frac{a}{8}\sum_{k=1}^{m}\sum_{l=1}^{n}\sum_{l'=1}^{n}\sum_{i=1}^{n}\Lambda_{il}^{\left(k\right)}\Lambda_{il'}^{\left(k\right)}x_{l'}x_{l}\right)\right).}
\end{array}
\]

Thus we get 
\[
\begin{array}{cl}
\Delta N & {\displaystyle =\frac{|x|^{2}}{N^{3}}\left[-3\left(N^{-4}\left(|x|^{4}+\frac{a^{2}}{16}\sum_{k=1}^{m}\sum_{l=1}^{n}\sum_{k'=1}^{m}\sum_{l'=1}^{n}\sum_{i=1}^{n}\Lambda_{il}^{\left(k\right)}\Lambda_{il'}^{\left(k'\right)}\frac{x_{l}}{|x|}\frac{x_{l'}}{|x|}z_{k}z_{k'}\right)\right)\right.}\\
 & +\left.\left(n+2+\frac{a}{8}\sum_{k=1}^{m}\sum_{l=1}^{n}\sum_{l'=1}^{n}\sum_{i=1}^{n}\Lambda_{il}^{\left(k\right)}\Lambda_{il'}^{\left(k\right)}\frac{x_{l}}{|x|}\frac{x_{l'}}{|x|}\right)\right].
\end{array}
\]

which can be represented as follows
\[
\begin{array}{cl}
\Delta N & {\displaystyle =(n-1)\frac{|x|^{2}}{N^{3}}}\\
 & {\displaystyle +\frac{|x|^{2}}{N^{3}}\left[-3\left(N^{-4}\left(-a|z|^{2}+\frac{a^{2}}{16}\sum_{k=1}^{m}\sum_{l=1}^{n}\sum_{k'=1}^{m}\sum_{l'=1}^{n}\sum_{i=1}^{n}\Lambda_{il}^{\left(k\right)}\Lambda_{il'}^{\left(k'\right)}\frac{x_{l}}{|x|}\frac{x_{l'}}{|x|}z_{k}z_{k'}\right)\right)\right.}\\
 & {\displaystyle \left.\frac{a}{8}\sum_{k=1}^{m}\sum_{l=1}^{n}\sum_{l'=1}^{n}\sum_{i=1}^{n}\Lambda_{il}^{\left(k\right)}\Lambda_{il'}^{\left(k\right)}\frac{x_{l}}{|x|}\frac{x_{l'}}{|x|}\right].}
\end{array}
\]

Hence, there exists a constant $B\in(0,\infty)$ such that
\[
|\Delta N|\leq B\frac{|x|^{2}}{N^{3}}
\]

\textbf{Remark:} If $a>0$ is small, $\Delta N\geq0$. For large $a$, in some
directions $\Delta N$ can be negative. 
\end{proof}

\end{document}